\documentclass[11pt]{article}

\usepackage{amsmath, amsthm, amssymb}
\usepackage{amsfonts,url}
\usepackage{epsfig,psfrag}
\usepackage{fullpage}

\newtheorem{theorem}{Theorem}[section]
\newtheorem{proposition}[theorem]{Proposition}
\newtheorem{lemma}[theorem]{Lemma}
\newtheorem{corollary}[theorem]{Corollary}
\newtheorem{definition}[theorem]{Definition}
\newtheorem{conjecture}[theorem]{Conjecture}
\newtheorem{question}[theorem]{Question}
\newcounter{Examplecount}
\newenvironment{example}[1][Example \arabic{Examplecount}.]{\begin{trivlist}
\item[\hskip \labelsep {\bfseries #1}]}{\end{trivlist}\stepcounter{Examplecount}}

\newcommand\beq{\begin{equation}}
\newcommand\eeq{\end{equation}}
\newcommand\bce{\begin{center}}
\newcommand\ece{\end{center}}
\newcommand\bea{\begin{eqnarray}}
\newcommand\eea{\end{eqnarray}}
\newcommand\ba{\begin{array}}
\newcommand\ea{\end{array}}
\newcommand\ben{\begin{enumerate}}
\newcommand\een{\end{enumerate}}
\newcommand\bit{\begin{itemize}}
\newcommand\eit{\end{itemize}}
\newcommand\brr{\begin{array}}
\newcommand\err{\end{array}}
\newcommand\bt{\begin{tabular}}
\newcommand\et{\end{tabular}}

\newcommand\nn{\nonumber}

\newcommand\ms{\medskip}

\renewcommand\S{{\mathcal S}}
\def\red{\operatorname{st}}
\def\Clo{\widehat{R}}
\def\Cl{R}
\def\cl{r}
\def\wh{\widehat}

\def\O{O}

\def\P{{\mathcal P}}
\def\I{\mathcal{I}}
\def\lext{\mathcal{L}}

\title{Clusters, generating functions and asymptotics\\ for consecutive patterns in permutations\thanks{Research
partially supported by NSF grant DMS-1001046.}}
\author{Sergi Elizalde~\thanks{Department of Mathematics, Dartmouth College, Hanover, NH 03755. E-mail: \texttt{sergi.elizalde@dartmouth.edu}}
\and
Marc Noy~\thanks{Departament
de Matem\`{a}tica Aplicada II, Universitat Polit\`{e}cnica de Catalunya,
Barcelona, Spain. E-mail: \texttt{marc.noy@upc.edu}.}}
\date{}

\begin{document}

\maketitle

\begin{abstract}
We use the cluster method to enumerate permutations avoiding consecutive patterns.
We reprove and generalize in a unified way several known results and obtain new ones,
including some patterns of length 4 and 5, as well as some infinite families of patterns of a given shape.
By enumerating linear extensions of certain posets, we find a differential equation satisfied by the inverse
of the exponential generating function counting occurrences of the pattern.
We prove that for a large class of patterns, this inverse is always an entire function.

We also complete the classification of consecutive patterns of length up to~6 into equivalence classes, proving a conjecture of Nakamura.
Finally, we show that the monotone pattern asymptotically dominates (in the sense that it is easiest to avoid)
all non-overlapping patterns of the same length, thus proving a conjecture of Elizalde and Noy for a positive fraction of all patterns.
\end{abstract}



\section{Introduction}

In this paper we use the cluster method of Goulden and Jackson in order to obtain new results on the enumeration of permutations avoiding consecutive patterns.
Recall that a permutation $\pi$ avoids a consecutive pattern $\sigma$ if no subsequence of adjacent entries of $\pi$ is in the same relative order as the entries of $\sigma$.
Given a pattern $\sigma$, the cluster method consists of counting partial permutations in which each element is involved in at least one occurrence of $\sigma$, the so-called clusters. By inclusion-exclusion, the enumeration of clusters provides the enumeration of permutations according to the number of occurrences of $\sigma$.

Counting clusters can be seen as counting linear extensions in a certain poset. For instance, if $\sigma$ is the monotone pattern, the corresponding poset is simply a chain, and counting linear extensions is a trivial task. In
fact, not only the monotone pattern $12\cdots m$ can be analyzed in this way, but also
the pattern $123 \cdots (s-1)(s+1)s(s+2)(s+3)\cdots m$
(Corollary~\ref{cor:124356gen}), and other related patterns, which we call chain patterns (Theorem~\ref{thm:chain}).
Another significant case is that of non-overlapping patterns $\sigma$, which are those for which two occurrences of $\sigma$ in a permutation cannot overlap in more than one position. The associated poset is not difficult to analyze when $\sigma_1=1$, in which case the number of permutations avoiding $\sigma\in\S_m$ depends only on the value $b=\sigma_m$, and we can derive a linear differential equation (Theorem~\ref{thm:1b}) satisfied by the inverse of the associated exponential generating function. A weaker version of this result was proved in \cite{EliNoy} using representations of permutations as binary trees.
 For chain patterns and non-overlapping patterns, the differential equations that we obtain
can also be deduced from the work of Khoroshkin and Shapiro~\cite{KS}.

A more intricate example is the pattern $1324$. This case was left open in~\cite{EliNoy} and cannot be solved with the techniques from~\cite{KS} either.
The number of linear extensions of the associated poset is related to the Catalan numbers, and we prove that the inverse of the generating function for this pattern satisfies a linear differential equation of order five with polynomial coefficients. Again, the technique can be extended to cover the general pattern $134\cdots (s+1)2(s+2)(s+3)\cdots m$ (Theorem~\ref{thm:1324gen}).
For other patterns of length $4$, namely $1423$ and $2143$, we find recurrence relations satisfied by their cluster numbers (Section~\ref{sec:four}), which already appeared in~\cite{Dot},
but we are not able to find closed solutions in terms of differential equations. In fact,
we conjecture that the inverse of the generating function for permutations avoiding $1423$ is not D-finite. If true,
this conjecture would give the first instance of a pattern with this property, and it would make a related conjecture of Noonan and Zeilberger for classical patterns less believable.

The present situation for small patterns is the following. Say that two patterns are equivalent if their numbers of occurrences in permutations have the same distribution.
There are two inequivalent patterns of length 3, already solved in \cite{EliNoy}.  There are seven inequivalent patterns of length 4, four of which are solved now, but we still do not have closed solutions for $1423$, $2143$ and $2413$.
There are 25 inequivalent patterns of length 5. Four of these are easily solved with the techniques from~\cite{EliNoy}, and we can now solve four additional ones,
namely $12435$, $12534$, $13254$ and $13425$.
The remaining 17 patterns (which include 2 non-overlapping ones) are unsolved in terms of closed solutions or differential equations.
For patterns of length 6, we prove four conjectures of Nakamura~\cite{Nak} regarding the equivalence of certain pairs, completing the classification into equivalence classes, and
proving that there are exactly 92 inequivalent patterns.

Regarding asymptotic enumeration of permutations avoiding a pattern, we
prove that the monotone pattern dominates all non-overlapping patterns of the same length (Theorem~\ref{thm:monvsnonover}), thus proving a special case of a conjecture by Elizalde and Noy~\cite{EliNoy}. It was shown by B\'ona~\cite{Bon} that the number of non-overlapping patterns is asymptotically a positive fraction of all patterns.
We also show that the inverse of the generating function of permutations according to the number of occurrences of a given pattern is an entire function in several important cases (Theorem~\ref{thm:entiregen}), but not for the pattern $2143$.

We conclude this introductory section with definitions and preliminaries needed in the rest of the paper.
In Sections~\ref{sec:monotone} and~\ref{sec:non-overlapping} we study monotone and non-overlapping patterns, and related patterns. Section~\ref{sec:1324} is devoted to the pattern $1324$ and generalizations, and Section~\ref{sec:four} to some other patterns of length 4. In Section \ref{sec:asymptotic} we present our asymptotic and analytic results. We end the paper with some open problems.

\subsection{Consecutive patterns}

Given a sequence of distinct positive integers $\tau = \tau_1 \cdots \tau_k$, we define the reduction $\red({\tau})$ as the permutation of length $k$ obtained by relabelling the elements of $\tau$ with $\{1,\dots,k\}$ so that the order relations among the elements remains the same. For instance $\red(46382) = 34251$.
Given permutations $\pi \in \S_n$ and $\sigma \in \S_m$, we say that $\pi$ contains $\sigma$ as a consecutive pattern if $\red(\pi_i \cdots \pi_{i+m-1}) = \sigma$ for some $i \in \{1,\dots,n-m+1\}$.
We denote by  $c_\sigma(\pi)$ the number of occurrences of $\sigma$ in $\pi$ as a consecutive pattern, and by
 $\alpha_n(\sigma)$ the number of permutations in $\S_n$ that avoid $\sigma$ as a consecutive pattern.
In the rest of the paper, the notions of occurrence, containment and avoidance always refer to consecutive patterns, even if it is not explicitly stated.

Let
$$P_{\sigma}(u,z)=\sum_{n\ge0} \sum_{\pi\in\S_n} u^{c_\sigma(\pi)}\frac{z^n}{n!}$$ be the bivariate exponential generating function for occurrences of $\sigma$ in permutations. It is convenient to define
$\omega_\sigma(u,z)=1/P_{\sigma}(u,z)$.
Note that $$P_\sigma(0,z)=\frac{1}{\omega_\sigma(0,z)}=\sum_{n\ge0} \alpha_n(\sigma)\frac{z^n}{n!}$$
is the generating function of permutations avoiding $\sigma$.
We drop the subscript $\sigma$ from $P$ and $\omega$ when the pattern is clear from the context.
If $\Sigma$ is a set of patterns, we define $P_\Sigma(u,z)$ and $\omega_\Sigma(u,z)$ similarly, where $u$ marks the total number of occurrences of all the patterns in $\Sigma$.

Two occurrences of $\sigma$ in a permutation may overlap in certain positions. This is a basic issue in what follows and motivates the following definition.
Let $\O_\sigma$ be the set of indices $i$ with $1\le i<m$ such that $\red(\sigma_{i+1}\sigma_{i+2}\dots\sigma_m)=\red(\sigma_1\sigma_2\dots \sigma_{m-i})$. Equivalently, $i\in\O_\sigma$ if there is some permutation in $\S_{m+i}$
where both its leftmost $m$ entries and its rightmost $m$ entries form occurrences of $\sigma$ (these occurrences overlap in exactly $m-i$ positions).
We call $\O_\sigma$ the set of overlaps of $\sigma$. Note that if $m\ge2$, then $m-1\in \O_\sigma$.

We say that two patterns $\sigma$ and $\tau$ are {\em strongly c-Wilf-equivalent}, or simply, that they fall in the same class, if $P_\sigma(u,z)=P_\tau(u,z)$.
While this implies that $\alpha_n(\sigma)=\alpha_n(\tau)$ for all $n$, it is an open question to determine whether the converse holds.
Given a pattern $\sigma_1 \cdots \sigma_m$, its reversal is
$\sigma_m \cdots \sigma_1$, and its complementation
is $(m+1-\sigma_1) \cdots (m+1 - \sigma_m)$.
Reversal and complementation do not change the equivalence class of a pattern. Patterns of small length were first studied in~\cite{EliNoy}.
It was shown that patterns of length three fall into two classes, represented by $123$ and $132$, and  the associated generating functions $P_\sigma(u,z)$ were computed explicitly.
Patterns of length four fall into seven classes, represented by $1234$, $2413$, $2143$, $1324$, $1423$, $1342$ and $1243$. Three of them, namely $1234$, $1342$ and $1243$, were solved in \cite{EliNoy} in terms of the generating functions.
Let us remark that the situation is quite different for classical patterns, that is, patterns that appear in non-necessarily consecutive positions of a permutation. For instance, all permutations of length 3 are Wilf-equivalent in the classical setting, due to a non-trivial bijection between permutations avoiding $123$ and those avoiding $132$.

\subsection{The cluster method}\label{sec:clustermethod}

The cluster method of Goulden and Jackson~\cite{GJ79,GJ} is a powerful method for enumerating words with respect to occurrences of certain substrings, based on inclusion-exclusion.
Several extensions and implementations of the method have been given in the literature,
most notably in~\cite{NZ}.
Let us now summarize an adaptation of the cluster method to the enumeration of permutations with respect to the number of occurrences of a consecutive pattern.
This adaptation has been recently used in~\cite{Dot}, and it has many similarities with a method of Mendes and Remmel~\cite{MR} based on the combinatorics of symmetric functions.

For fixed $\sigma\in\S_m$, a $k$-cluster of length $n$ with respect to $\sigma$ is a pair
$(\pi,(i_1,i_2,\dots,i_k))$ where \bit
\item $\pi\in\S_n$,
\item $1=i_1<i_2<\dots<i_k=n-m+1$,
\item for each $1\le j\le k$, $\red(\pi_{i_j}\pi_{i_j+1}\dots\pi_{i_j+m-1})=\sigma$,
\item for each $1\le j\le k-1$, $i_{j+1}\le i_j+m-1$.
\eit
In other words, the $i_j$ are starting positions of occurrences of $\sigma$ in $\pi$,
all the entries of $\pi$ are in at least one of the marked occurrences, and neighboring marked occurrences overlap. Note that $i_{j+1}-i_j\in\O_\sigma$ (the overlap set) for all $j$, and that $\pi$ may have more than $k$ occurrences of $\sigma$.
Sometimes we  write $(\pi;i_1,i_2,\dots,i_k)$ instead of $(\pi,(i_1,i_2,\dots,i_k))$.
For example, if $\sigma=1324$, then $(142536879;1,3,6)$ is a $3$-cluster of length $9$, since
$1425, 2536$ and $6879$ are occurrences of $\sigma$ and all the entries are in at least one of the occurrences. Notice that the $1425$ and $2536$ overlap in two positions, whereas $ 2536$ and $6879$ overlap only in one position.

Let $\cl_{n,k}$ be the number of $k$-clusters of length $n$ with respect to a fixed $\sigma$. We denote by
$$\Cl_\sigma(t,z)=\sum_{n,k}\cl_{n,k}t^k\frac{z^n}{n!}$$ the exponential generating function (EGF for short) for clusters, and by
$$\Clo_\sigma(t,x)=\sum_{n,k}\cl_{n,k}t^k x^n$$ the corresponding ordinary generating function (OGF for short).

The following theorem, which is an adaptation of \cite[Theorem 2.8.6]{GJ} to the case of permutations, expresses the EGF for occurrences of $\sigma$ as a consecutive pattern in permutations in terms of the EGF for clusters. We include a short proof for completeness.

\begin{theorem}[\cite{GJ}]\label{thm:GJ} For any pattern $\sigma$ we have
$$\omega_\sigma(u,z)=1-z-\Cl_\sigma(u-1,z).$$
In particular, the EGF for $\sigma$-avoiding permutations is $$P_\sigma(0,z)=\frac{1}{\omega_\sigma(0,z)}=\frac{1}{1-z-\Cl_\sigma(-1,z)}.$$
\end{theorem}

\begin{proof}
A permutation $\pi$ can be seen as a sequence of consecutive blocks, where each block consists either of elements all belonging to an occurrence of $\sigma$, or of an element not belonging to any
occurrence of $\sigma$. The first kind of block is a cluster and is encoded by the generating function $R_\sigma(-1,z)$ (since setting $u=-1$ marks, by inclusion-exclusion, the exact number of occurrences of $\sigma$), and the single elements are encoded by $z$.
Notice that the fact that occurrences in a cluster overlap guarantees the uniqueness of the decomposition. Finally, the sequence construction corresponds to $1/(1-(z+R_\sigma(-1,z)))$, thus proving the result.
\end{proof}

We usually denote $\wh{A}(t,x)=1-x-\Clo(t,x)$ and $A(t,z)=1-z-\Cl(t,z)$, so that $\omega_\sigma(u,z)=A(u-1,z)$.

Theorem~\ref{thm:GJ} reduces the study of the distribution of occurrences of a pattern in permutations to computing the cluster numbers $\cl_{n,k}$. We  now show that these numbers can be expressed in terms of linear extensions of certain posets.
Fix $\sigma\in\S_m$. For given $n,k$, let
\beq\label{eq:I}\I^\sigma_{n,k}=\{(i_1,i_2,\dots,i_k) : i_1=1, i_k=n-m+1, \mbox{ and } i_{j+1}-i_j\in\O_\sigma \mbox{ for }1\le j\le k-1\},\eeq
and fix $(i_1,\dots,i_k)\in\I^\sigma_{n,k}$. A permutation $\pi\in\S_n$ has the property that
$(\pi;i_1,\dots,i_k)$ is a $k$-cluster of length $n$ if and only if, for each $1\le j\le k$,
\beq\label{eq:occurrence}\red(\pi_{i_j}\pi_{i_j+1}\dots\pi_{i_j+m-1})=\sigma.\eeq
If we denote by $\varsigma\in\S_m$ the inverse of $\sigma$, so that $\varsigma_\ell$ is the position of $\ell$ in $\sigma$, then~\eqref{eq:occurrence} is equivalent to
\beq\label{eq:occurrence2}\pi_{\varsigma_1+i_j-1}<\pi_{\varsigma_2+i_j-1}<\dots<\pi_{\varsigma_m+i_j-1}.\eeq
The conditions~\eqref{eq:occurrence2} for $1\le j\le k$ define a partial order on the set $\{\pi_1,\pi_2,\dots,\pi_n\}$. We denote the corresponding partially ordered set (poset)
by $P^\sigma_{n,i_1,\dots,i_k}$ . If we denote by $\lext(P)$ the set of linear extensions
(i.e., compatible linear orders) of $P$, then it follows that
$(\pi;i_1,\dots,i_k)$ is a $k$-cluster of length $n$ with respect to $\sigma$ if and only if $\pi\in\lext(P^\sigma_{n,i_1,\dots,i_k})$.
We denote by $\P^\sigma_{n,k}$ the multiset of such posets for all values of $(i_1,\dots,i_k)\in\I^\sigma_{n,k}$.
Note that some posets in $\P^\sigma_{n,k}$ can appear with multiplicity, as in the cases discussed in Section~\ref{sec:monotone}.
Alternatively, we could mark the elements $\pi_{i_1},\pi_{i_2},\dots,\pi_{i_k}$ in $P^\sigma_{n,i_1,\dots,i_k}$, to ensure that
all the posets in $\P^\sigma_{n,k}$ are different as marked posets.
We have that \beq\label{eq:clinext}\cl_{n,k}=\sum_{P\in\P^\sigma_{n,k}}|\lext(P)|.\eeq
We also define the multisets $\P^\sigma_n=\bigcup_{k\ge1} \P^\sigma_{n,k}$ and $\P^\sigma=\bigcup_n \P^\sigma_{n}$.

Finally, note that for the reversal or complementation of a pattern $\sigma$, the corresponding clusters are also reversed or complemented, the set $\O_\sigma$ does not change, and the posets that we obtain are isomorphic to those for $\sigma$.

\subsection{Ordinary and exponential generating functions}

Here we describe a tool that we use to switch between the OGF and the EGF of a sequence.
Let $L$ be the linear operator on formal power series such that $L(x^k)=\frac{z^k}{k!}$ for all $k\ge0$.
\begin{lemma}\label{lem:L}
Let $\wh{A}(x)=\sum_{n\ge0}a_n{x^n}$ be an OGF and let $A(z)=\sum_{n\ge0}a_n\frac{z^n}{n!}$ be the corresponding EGF. Let $I$ denote the integral operator with respect to $z$, that is, $I F(z)= \int_0^z F(v)\, dv$, and let $j\ge0$. Then
\begin{enumerate}
\item $L(x^{j}\wh{A})=I^{j}A$;
\item $L(x^{j+1}\wh{A}\,')=I^{j}(zA')$;
\item $L(x^j\wh{A}^{(j)})=z^jA^{(j)}$.
\end{enumerate}
\end{lemma}

\begin{proof}
All the properties are easy to prove. For instance, claim~1 follows from
$$L(x^j \wh{A}) =
L\left(\sum a_n x^{n+j}\right) = \sum a_n \frac{z^{n+j}}{(n+j)!} = I^j A.
$$
\end{proof}

Given a linear differential equation for $\wh{A}(x)$ with polynomial coefficients, Lemma~\ref{lem:L} will be used to
obtain a linear differential equation for $A(z)$ with polynomial coefficients.

Let us finish this section with a remark about notation. Throughout the paper we deal with differential equations for multivariate generating functions. All the derivatives that appear are always partial derivatives with respect to $z$. Similarly, initial conditions are for $z$.

\section{Monotone and related patterns}\label{sec:monotone}

\subsection{The pattern $\sigma=12\dots m$}

For the monotone pattern $\sigma=12\dots m$, a differential equation satisfied by $\omega_\sigma(u,z)$ was given in~\cite[Theorem 3.1]{EliNoy}. The proof is based on
representations of permutations as increasing binary trees.

\begin{theorem}[\cite{EliNoy}]\label{thm:monotone}
Let $m\ge3$, let $\sigma=12\dots m$, and let $\omega(z):=\omega_\sigma(u,z)$. Then $\omega$ is the solution of
\beq\label{eq:omegamon}\omega^{(m-1)}+(1-u)(\omega^{(m-2)}+\dots+\omega'+\omega)=0\eeq
with $\omega(0)=1,\omega'(0)=-1,\omega^{(i)}(0)=0$ for $2\le i\le m-2$.
\end{theorem}

To warm up for the more complicated patterns in upcoming sections, let us give an alternative proof of Theorem~\ref{thm:monotone} using the cluster method. It is clear that
$O_\sigma=\{1,2,\dots,m-1\}$, so for $\pi\in\S_n$, $(\pi;i_1,\dots,i_k)$ is a $k$-cluster with respect to $\sigma$ if and only if $\pi_{1}<\pi_{2}<\dots<\pi_{n}$ and $1\le i_{j+1}-i_j\le m-1$ for all $j$.
It follows that the OGF for the cluster numbers is
$$\Clo_\sigma(t,x)=\frac{tx^{m}}{1-t(x+x^2+\dots+x^{m-1})},$$
and so \beq\label{eq:Clomon}\wh{A}(t,x)=1-x-\Clo_\sigma(t,x)=\frac{1-x-tx}{1-t(x+x^2+\dots+x^{m-1})}.\eeq
Clearing denominators in~(\ref{eq:Clomon}), applying the transformation $L$, and using Lemma~\ref{lem:L}, we get that
\beq\label{eq:Imonotone}(1-t(I+I^2+\dots+I^{m-1}))A(t,z)=1-z-tz.\eeq
Differentiating $m-1$ times we obtain a differential equation for $A(t,z)$:
$$A^{(m-1)}-t(A^{(m-2)}+\dots+A'+A)=0,$$
with initial conditions
$A(0) = 1, A'(0) = -1$, and $A^{(i)}(0) = 0$ for $2\le i\le m-2$.
Equation~(\ref{eq:omegamon}) is now obtained making the substitution $t=u-1$ and using Theorem~\ref{thm:GJ}. Note that for $m=2$, the right hand side of~\eqref{eq:Imonotone} does not cancel, but it produces a term $-1-t$, from where one can recover the well-known generating function for the Eulerian numbers.

In fact, we see from equation~(\ref{eq:Clomon}) that
$$1-x-\Clo_\sigma(-1,x)=\frac{1-x}{1-x^m}=\sum_{j\ge0}x^{jm}-\sum_{j\ge0}x^{jm+1},$$
so
\beq\label{eq:omega_monotone}\omega_\sigma(0,z)=\sum_{j\ge0}\frac{z^{jm}}{(jm)!}-\sum_{j\ge0}\frac{z^{jm+1}}{(jm+1)!},\eeq
recovering a formula from~\cite{GJ}.
Equivalently,
$$\omega_\sigma(0,z)=\frac{1}{m} \sum_{j=0}^{m-1} \left(1-\frac{1}{\lambda_j}\right)e^{\lambda_j z},$$
where $\lambda_j=e^\frac{2\pi ij}{m}$ are the $m$th roots of unity.

\subsection{Chain patterns}

The ideas in the above proof of Theorem~\ref{thm:monotone} extend to any pattern $\sigma$ for which the poset satisfied by the entries of $\pi$ in every cluster is a chain. Recall that a poset is a chain if it is a linear order, that is, all its elements are comparable.
\begin{definition}
We say that $\sigma\in\S_m$ is a {\em chain pattern} if all the posets in $\P^\sigma$ are chains.
\end{definition}
Examples of chain patterns are given in Subsection~\ref{sec:chainexamples}. The following result shows that the condition of being a chain pattern is quite restrictive.

\begin{lemma}\label{lem:chain}
Let $m\ge3$, and let $\sigma\in\S_m$ be a chain pattern. Then $\sigma$ (or one of the permutations obtained by applying the reversal and/or complementation operations) satisfies that $\sigma_1=1$, $\sigma_2=2$ and $\sigma_m=m$. In particular, $m-2\in\O_\sigma$.
\end{lemma}

\begin{proof}
Recall that always $m-1\in\O_\sigma$. By hypothesis, the poset $P^\sigma_{2m-1,1,m}$, corresponding to clusters consisting of two occurrences of $\sigma$ overlapping in one position, is a chain.
Denoting by $\varsigma\in\S_m$ the inverse of $\sigma$, this poset is determined by the inequalities
$\pi_{\varsigma_1}<\pi_{\varsigma_2}<\dots<\pi_{\varsigma_m}$ and
$\pi_{\varsigma_1+m-1}<\pi_{\varsigma_2+m-1}<\dots<\pi_{\varsigma_m+m-1}$.
Since $\pi_m$ appears in both lists, it must be the rightmost element of one list and the leftmost one of the other, from where $\{\sigma_1,\sigma_m\}=\{1,m\}$.
By applying reversal if necessary, we can assume that $\sigma_1=1$ and $\sigma_m=m$. Since $\sigma_1<\sigma_2$ and $\sigma_{m-1}<\sigma_m$, we have that $m-2\in\O_\sigma$.

Consider now the poset $P^\sigma_{2m-2,1,m-1}$, corresponding to clusters consisting of two occurrences of $\sigma$ overlapping in two positions, which again is a chain. This poset is determined by the inequalities
$\pi_{1}<\pi_{\varsigma_2}<\dots<\pi_{\varsigma_{m-1}}<\pi_{m}$ and
$\pi_{m-1}<\pi_{\varsigma_2+m-2}<\dots<\pi_{\varsigma_{m-1}+m-2}<\pi_{2m-2}$. For this poset not to have incomparable elements, we must have either $\varsigma_{m-1}=m-1$ or $\varsigma_2+m-2=m$, which is equivalent to
$\sigma_{m-1}=m-1$ or $\sigma_2=2$.
\end{proof}

Now we state the generalization of Theorem~\ref{thm:monotone} to chain patterns.

\begin{theorem}\label{thm:chain}
Let $m\ge3$, and let $\sigma\in\S_m$ be a chain pattern. Let $\omega(z):=\omega_{\sigma}(u,z)$. Then $\omega$ is the solution of
\beq\label{eq:omegachain}\omega^{(m-1)}+(1-u)\sum_{d\in\O_\sigma}\omega^{(m-d-1)}=0\eeq
with $\omega(0)=1,\omega'(0)=-1,\omega^{(i)}(0)=0$ for $2\le i\le m-2$.
\end{theorem}

\begin{proof}
For each fixed tuple $(i_1,i_2,\dots,i_k)\in\I^\sigma_{n,k}$, as defined in~\eqref{eq:I}, there is a unique permutation $\pi\in\S_n$ such that $(\pi;i_1,\dots,i_k)$ is a cluster with respect to $\sigma$.
Indeed, $\pi$ is the unique linear extension of the poset $P^\sigma_{n,i_1,i_2,\dots,i_k}$, as defined in Section~\ref{sec:clustermethod}, which is a chain by hypothesis.
Thus, the number of $k$-clusters of length $n$ is the number of elements of $\I^\sigma_{n,k}$, so the corresponding OGF is
$$\Clo_\sigma(t,x)=\frac{tx^{m}}{1-t \sum_{d\in\O_\sigma} x^{d}},$$
and so $\wh{A}(t,x)=1-x-\Clo_\sigma(t,x)$ satisfies the equation
\beq\label{eq:Achain}\wh{A}(t,x)-t \sum_{d\in\O_\sigma} x^{d}\wh{A}(t,x)+t\sum_{d\in\O_\sigma}(x^d-x^{d+1})+tx^m+x-1=0.\eeq
Now we apply the transformation $L$ to~\eqref{eq:Achain} and use Lemma~\ref{lem:L}. Differentiating $m-1$ times and using that $\{m-2,m-1\}\subset\O_\sigma$, which we showed in Lemma~\ref{lem:chain},
we get a differential equation for $A(z)=A(t,z)$:
$$A^{(m-1)}-t\sum_{d\in\O_\sigma} A^{(m-d-1)}=0,$$
with initial conditions
$A(0) = 1, A'(0) = -1$, and $A^{(i)}(0) = 0$ for $2\le i\le m-2$.
A differential equation for $\omega$ is now obtained making the substitution $t=u-1$ and using Theorem~\ref{thm:GJ}.
\end{proof}

\subsection{Examples}\label{sec:chainexamples}

A good example is the pattern $123\dots(s-1)(s+1)s(s+2)(s+3)\dots m$, for arbitrary $s\ge3$ and $m\ge s+2$. By reversal and complementation, we can assume without loss of generality that $m-s\le s$. In this case, $O_\sigma=\{s,s+1,\dots,m-1\}$.

\begin{corollary}\label{cor:124356gen}
Let $s\ge3$ and $2\le m-s\le s$, let $\sigma=123\dots(s-1)(s+1)s(s+2)(s+3)\dots m$, and let $\omega(z):=\omega_\sigma(u,z)$. Then $\omega$ is the solution of
$$\omega^{(m-1)}+(1-u)(\omega^{(m-s-1)}+\dots+\omega'+\omega)=0$$
with
$\omega(0)=1,\omega'(0)=-1,\omega^{(i)}(0)=0$ for $2\le i\le m-2$.
\end{corollary}

Corollary~\ref{cor:124356gen} for $s=3$ and $m=5$ states that $\omega_{12435}(u,z)$ satisfies the differential equation
$$\omega^{(4)}+(1-u)(\omega'+\omega)=0.$$ We can write
$$P_{12435}(u,z)=\left(\sum_\alpha\frac{\alpha^3-\alpha^2+1-u}{4\alpha^3+1-u}\,e^{\alpha z}\right)^{-1},$$
where the sum is over the four roots $\alpha$ of the characteristic polynomial $x^4+(1-u)(x+1)$.
For $t=-1$, the OGF for the cluster numbers $\Clo_{12435}(-1,x)=-x^5/(1+x^3+x^4)$ has been obtained automatically by Baxter, Nakamura and Zeilberger's
Maple package~\cite{BNZ}.

A more general family of chain patterns can be obtained as follows.
Given $\sigma\in\S_m$, let $r\ge0$ be the largest index such that $\sigma_1\sigma_2\dots\sigma_r=12\dots r$,
let $s\ge0$ be the largest such that $\sigma_{m-s+1}\dots\sigma_{m-1}\sigma_m=(m-s+1)\dots(m-1)m$,
let $a\ge1$ be the largest such that $\sigma_1\sigma_2\dots\sigma_a$ is increasing,
let $b\ge1$ be the largest such that $\sigma_{m-b+1}\dots\sigma_{m-1}\sigma_m$ is increasing, and let $c=\min\{a,b\}$.

\begin{corollary}\label{cor:chain2}
Let $\sigma\in\S_m\setminus\{12\dots m\}$, and let $r,s,a,b,c$ be defined as above.
Suppose that $r,s\ge1$ and $r+s\ge c+1$, and that $\O_\sigma\cap \{1,2,\dots,m-c-1\}=\emptyset$ (i.e., $\sigma$ can
only overlap with itself at the initial and final increasing runs).
Let $\omega(z):=\omega_{\sigma}(u,z)$. Then $\omega$ is the solution of
$$\omega^{(m-1)}+(1-u)(\omega^{(c-1)}+\dots+\omega'+\omega)=0$$
with $\omega(0)=1,\omega'(0)=-1,\omega^{(i)}(0)=0$ for $2\le i\le m-2$.
\end{corollary}

In~\cite{Nak}, Nakamura conjectures that $123546$ and $124536$ are strongly c-Wilf-equivalent, that is, $P_{123546}(u,z)=P_{124536}(u,z)$. We can prove this conjecture using Corollary~\ref{cor:chain2}.
Indeed, for $\sigma=123546$ we have $r=3$, $s=1$, $a=4$, $b=2$, $c=2$, $\O_\sigma=\{4,5\}$, while for $\sigma=124536$ all the parameters are the same except that $r=2$. Thus,
both $\omega_{123546}(u,z)$ and $\omega_{124536}(u,z)$ satisfy $$\omega^{(5)}+(1-u)(\omega'+\omega)=0$$
with the same initial conditions.

\section{Non-overlapping and related patterns}\label{sec:non-overlapping}

\subsection{Non-overlapping patterns}

We say that $\sigma\in\S_m$ is non-overlapping if $\O_\sigma=\{m-1\}$, that is, two occurrences of $\sigma$ cannot overlap in more than one position. We assume in this section that $m\ge2$. In~\cite{Bon}, B\'ona gives asymptotic estimates on the
number of non-overlapping permutations of length $m$, showing in particular that there is a positive fraction of them.

In~\cite[Theorem 3.2]{EliNoy}, the authors used binary trees to enumerate occurrences of non-overlapping patterns of the form
$\sigma=12\dots(b-1)\tau b$, where $2\le b<m$, and $\tau$ is any permutation of $\{b+1,b+2,\dots,m\}$. In fact, the formula holds for slightly more general patterns, as noted in~\cite{Dot},
namely non-overlapping patterns with $\sigma_1=1$.

\begin{theorem}[a weaker version appears in \cite{EliNoy}]\label{thm:1b}
Let $\sigma\in\S_m$ be a non-overlapping pattern with $\sigma_1=1$, let $b=\sigma_m$, and let $\omega(z):=\omega_\sigma(u,z)$. Then $\omega$ is the solution of
\beq\label{eq:omega1b}\omega^{(b)}+(1-u)\frac{z^{m-b}}{(m-b)!}\omega'=0\eeq
with $\omega(0)=1,\omega'(0)=-1,\omega^{(i)}(0)=0$ for $2\le i\le b-1$.
\end{theorem}

Here we give a proof of Theorem~\ref{thm:1b} using the cluster method, different from the proof in~\cite{EliNoy}. It will be convenient to describe the posets in $\P_\sigma$ for an arbitrary non-overlapping pattern $\sigma$.
Let $a=\sigma_1$ and $b=\sigma_m$, and note that $a\neq b$. Without loss of generality we can assume that $a<b$, since $\sigma$ and its reversal have the same cluster numbers.
Since $\sigma$ is non-overlapping, a $k$-cluster $(\pi;i_1,i_2,\dots,i_k)$ with respect to $\sigma$ has length $n=k(m-1)+1$ and satisfies $i_{j+1}-i_j=m-1$ for all $j$. In the rest of this section, let $n=k(m-1)+1$.

We have that $\I^\sigma_{n,k}=\{(1,m,2m-1,3m-2,\dots,(k-1)(m-1)+1)\}$, and $\P^\sigma_{n}$ consists of exactly one poset $P$. To find the shape of $P$, denote by $\varsigma\in\S_m$ the inverse of $\sigma$,
and observe that the first $m$ entries of $\pi$ must satisfy $\pi_{\varsigma_1}<\pi_{\varsigma_2}<\dots<\pi_{\varsigma_m}$ to form an occurrence of $\sigma$. Similarly, the entries in positions between $m$ and $2m-1$ must satisfy
$\pi_{\varsigma_1+m-1}<\pi_{\varsigma_2+m-1}<\dots<\pi_{\varsigma_m+m-1}$. Note that $\pi_{\varsigma_b}=\pi_m=\pi_{\varsigma_a+m-1}$ appears in both lists of inequalities.
Repeating this argument for each of the $k$ occurrences of $\sigma$ in the cluster, we see that $P$ is the poset in Figure~\ref{fig:poset1overlap}(i). By equation~\eqref{eq:clinext}, $\cl_{n,k}$ is the number of linear extensions
of $P$.
Figure~\ref{fig:poset1overlap}(ii) shows another drawing of $P$, containing a central vertical chain with $m+(b-a)(k-1)$ elements, to which there are $k-1$ chains pointing upwards with $m-b$ elements each and $k-1$ chains pointing downwards with $a-1$ elements each.

\begin{figure}[htb]
\centering
\psfrag{a-1}{$a-1$}\psfrag{b-1}{$b-1$}\psfrag{m-a}{$m-a$}\psfrag{m-b}{$m-b$}\psfrag{b-a}{$b-a$}
\bt{ccc}\includegraphics[height=8cm]{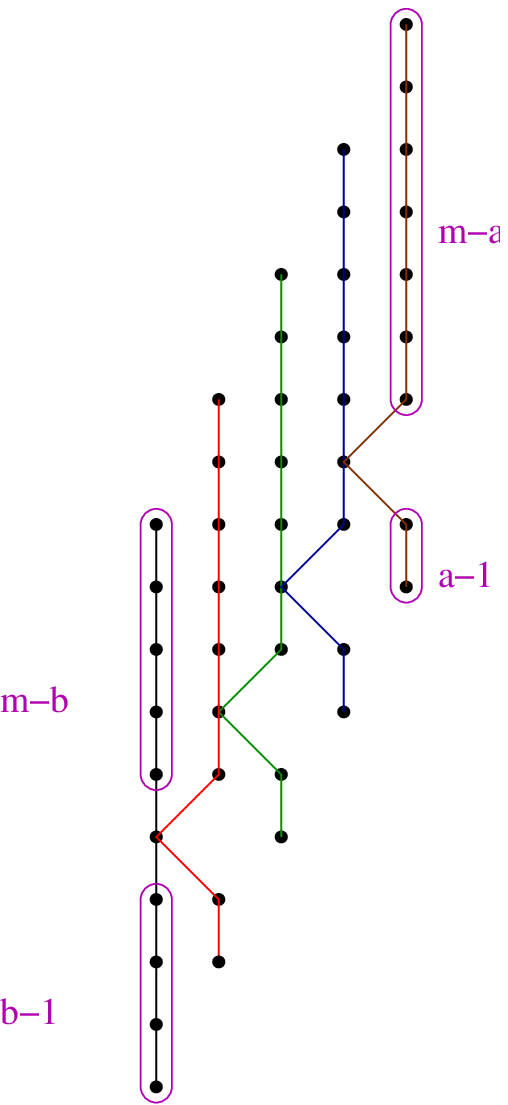}&\hspace*{15mm}&\includegraphics[height=8cm]{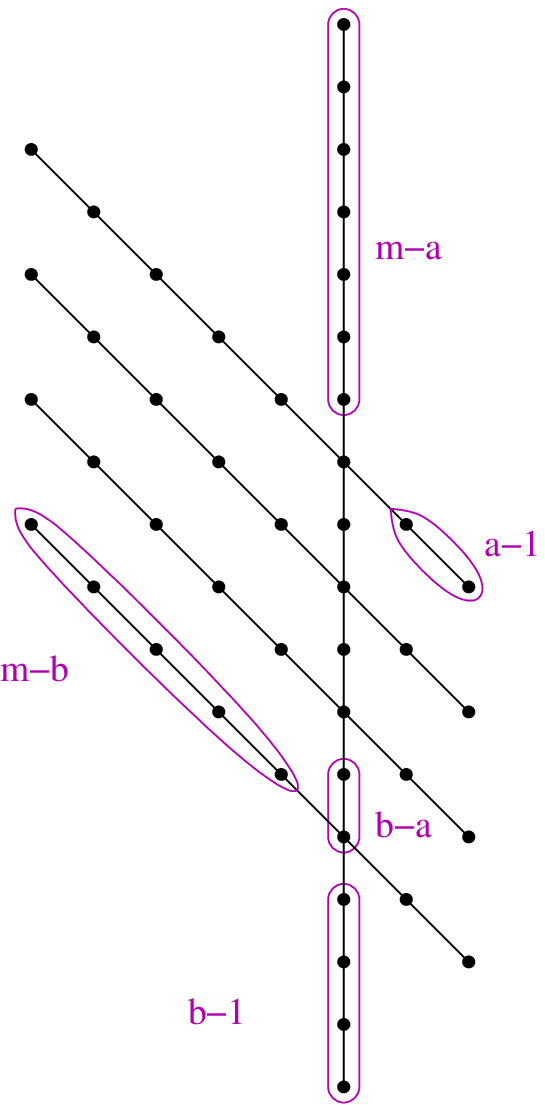}\\
(i)&&(ii)\et
\caption{\label{fig:poset1overlap} Two drawings of the unique poset in $\P^\sigma_n$ for a non-overlapping pattern $\sigma\in\S_m$ with $\sigma_1=a$ and $\sigma_m=b$.}
\end{figure}

It is clear from the above description that the unique poset in $\P^\sigma_{n}$ depends only on $\sigma_1$ and $\sigma_m$ but not on the other entries of $\sigma$, and hence so do the cluster numbers. It follows, using Theorem~\ref{thm:GJ},
that for a non-overlapping pattern $\sigma\in\S_m$, the generating function $P_\sigma(u,z)$ depends only on $\sigma_1$ and $\sigma_m$. This fact
was recently observed by Dotsenko and Khoroshkin~\cite{Dot}, and independently by Duane and Remmel~\cite{DR}.

\begin{proof}[Proof of Theorem~\ref{thm:1b}]
When $\sigma_1=1$, the poset $P$ in Figure~\ref{fig:poset1overlap}(i) can be decomposed by peeling off the first $m-1$ entries of $\pi$, which leaves another poset $P'$ of the same shape as $P$ whose bottom element corresponds to the entry $\pi_m$ (because $a=1$).
Of the peeled off elements, $b-1$ are smaller than $\pi_m$, and $m-b$ are larger than $\pi_m$. To produce a linear extension of $P$ given a linear extension of $P'$, one needs to choose the relative order of these $m-b$ elements
with respect to the elements of $P'$ other than $\pi_m$. This is equivalent to choosing an $(m-b)$-element multiset of the $n$ places where these elements could be inserted in a linear extension of $P'$, where $n$ is the size (number of elements) of $P'$. If $\Clo_\sigma(t,x)=\sum_{n}\cl_n^\sigma(t) x^n$ is the generating function for linear extensions of $P'$,
the generating function for such linear extensions with $m-b$ chosen places is
$$\sum_{n}\binom{n-1+m-b}{m-b}\cl_n^\sigma(t) x^n=\frac{x}{(m-b)!}\frac{\partial^{m-b}}{\partial x^{m-b}}\left(x^{m-b-1}\Clo_\sigma(t,x)\right).$$
Separating the case in which $P$ has size $m$ and applying the above decomposition to the other cases, we get a differential equation for $\Clo_\sigma$:
$$\Clo_\sigma(t,x)=tx^m+\frac{tx^m}{(m-b)!}\frac{\partial^{m-b}}{\partial x^{m-b}}\left(x^{m-b-1}\Clo_\sigma(t,x)\right).$$
We can easily turn it into a differential equation for $\wh{A}(t,x)=1-x-\Clo_{\sigma}(t,x)$, which then, applying the operator $L$ and using Lemma~\ref{lem:L}, becomes
$$A(t,z)=1-z+\frac{t}{(m-b)!}I^b\left(z^{m-b}\frac{\partial}{\partial z}A(t,z)\right),$$
where $A(t,z)=1-z-\Cl_\sigma(t,z)$. Differentiating $b$ times, making the substitution $t=u-1$, and using Theorem~\ref{thm:GJ}, we get the differential equation~(\ref{eq:omega1b}) for $\omega_\sigma(u,z)$.
\end{proof}

For the special case of $b=2$, Theorem~\ref{thm:1b} gives the explicit expression $$\omega_\sigma(u,z)=1-\int_0^z e^{(u-1)\frac{v^{m-1}}{(m-1)!}} dv,$$
which can be easily checked to be the solution of equation~\eqref{eq:omega1b}.

\subsection{The patterns $12534$ and  $13254$}

Here we consider some patterns that are neither non-overlapping nor chain patterns, yet they can be solved using similar ideas to those in the proof of Theorem~\ref{thm:1b}. The patterns $12534$ and  $13254$ are also special cases of~\cite[Corollary~3.9]{KS}.

\begin{proposition}\label{prop:12534}
Let $\omega(z)=\omega_{12534}(u,z)$. Then $\omega$ is the solution of
$$\omega^{(4)}+(1-u)z(\omega''+\omega')=0$$
with $\omega(0) = 1, \omega'(0) = -1, \omega''(0) = \omega^{(3)}(0) = 0$.
\end{proposition}

\begin{proof}
Let us  find an equation satisfied by the OGF for the cluster numbers $\Clo_{12534}(t,x)$.
 Let $(\pi;i_1,\dots,i_k)$ be a $k$-cluster of length $n$ with respect to $12534$. If $k=1$, its contribution to the generating function $\Clo_{12534}(t,x)$ is $tx^5$. Suppose now that $k\ge2$.
Clearly $i_2\in\{4,5\}$, since $i_2-i_1\in\O_{12534}=\{3,4\}$.
If $i_2=5$, then the poset $P^{12534}_{n,i_1,\dots,i_k}$ can be decomposed as a chain $\pi_1<\pi_2<\pi_4<\pi_5<\pi_3$, where $\pi_5$ is identified with the bottom element of a poset $P'\in\P^{12534}_{n-4,k-1}$.
Thus, the contribution of all clusters with $i_2=5$ to the generating function is $tx^5 \frac{\partial}{\partial x}\Clo_{12534}(t,x)$, where the derivative arises when determining the relative order of $\pi_3$ with respect to the elements in $P'$.
Similarly, if $i_2=4$, then $P^{12534}_{n,i_1,\dots,i_k}$ consists of a chain $\pi_1<\pi_2<\pi_4<\pi_5<\pi_3$, where $\pi_4$ and $\pi_5$ are identified with the bottom two elements of a poset in $\P^{12534}_{n-3,k-1}$.
The contribution of all clusters with $i_2=4$ is $tx^5 \frac{\partial}{\partial x}\frac{\Clo_{12534}(t,x)}{x}$.

It follows that $\Clo_{12534}(t,x)$ satisfies the differential equation
$$\Clo_{12534}(t,x)=tx^5+tx^5 \frac{\partial}{\partial x}\left(\left(1+\frac{1}{x}\right)\Clo_{12534}(t,x)\right).$$
Writing $\wh{A}(t,x)=1-x-\Clo_{12534}(t,x)$, applying the operator $L$, differentiating 4 times, making the substitution $t=u-1$,
and using Theorem~\ref{thm:GJ}, we obtain the stated differential equation for $\omega_{12534}(u,z)$.
\end{proof}

Nakamura~\cite{Nak} conjectures that $123645$ and $124635$ are strongly c-Wilf-equivalent. We can prove this with an argument analogous to the above proof of Proposition~\ref{prop:12534}.
Indeed, the posets $P^{\sigma}_{n,i_1,\dots,i_k}$ are isomorphic for the two patterns, and both $\omega_{123645}(u,z)$ and $\omega_{124635}(u,z)$ satisfy the differential equation $$\omega^{(5)}+(1-u)z(\omega''+\omega')=0$$
with $\omega(0) = 1, \omega'(0) = -1, \omega^{(i)}(0) =0$ for $2\le i\le 4$.

\begin{proposition}\label{prop:13254}
Let $\omega(z)=\omega_{13254}(u,z)$. Then $\omega$ is the solution of
$$\omega^{(4)}+(1-u)(\omega''+z\omega')=0$$
with $\omega(0) = 1, \omega'(0) = -1, \omega''(0) = \omega^{(3)}(0) = 0$.
\end{proposition}

\begin{proof}
In this case, $\O_{13254}=\{2,4\}$.
A $k$-cluster with respect to $13254$ where $i_{j+1}-i_j=2$ for all $j$ is a chain of length $5+2(k-1)$. Indeed, the relationships $\pi_1<\pi_3<\pi_2<\pi_5<\pi_4$ and $\pi_3<\pi_5<\pi_4<\pi_7<\pi_6$ imply
that $\pi_1<\pi_3<\pi_2<\pi_5<\pi_4<\pi_7<\pi_6$, and so on. Thus, the OGF for such clusters is $\frac{tx^5}{1-tx^2}$.

Consider now an arbitrary cluster $(\pi;i_1,\dots,i_k)$ with respect to $13254$. If it is not of the type considered above, let $j$ be the smallest index such that $i_{j+1}-i_j=4$. The poset $P^{13254}_{n,i_1,\dots,i_k}$
consists of a chain $\pi_1<\pi_3<\pi_5<\pi_4<\pi_7<\pi_6<\dots<\pi_{i_{j+1}}<\pi_{i_{j+1}-1}$, where $\pi_{i_{j+1}}$ is identified with the bottom element of a poset $P'\in\P^{13254}$. Note that in such a cluster,
the relative order of $\pi_{i_{j+1}-1}$ with the other elements of $P'$ is arbitrary. From this decomposition, the following differential equation for $\Clo_{13254}(t,x)$ follows.
$$\Clo_{13254}(t,x)=\frac{tx^5}{1-tx^2}\left(1+\frac{\partial}{\partial x}\Clo_{13254}(t,x)\right).$$
The rest of the proof is analogous to that of Proposition~\ref{prop:12534}.
\end{proof}

We can use an argument analogous to the above proof of Proposition~\ref{prop:13254} to prove another of Nakamura's conjectures~\cite{Nak}, namely that $132465$ and $142365$ are strongly c-Wilf-equivalent.
The key observation is that if $\sigma$ is either of these patterns, a $k$-cluster with respect to $\sigma$ where $i_{j+1}-i_j=3$ for all $j$ is a chain of length $6+3(k-1)$.
It follows that the posets $P^{\sigma}_{n,i_1,\dots,i_k}$ are isomorphic for the two patterns, and that both $\omega_{132465}(u,z)$ and $\omega_{142365}(u,z)$ satisfy the differential equation $$\omega^{(5)}+(1-u)(\omega''+z\omega')=0$$
with $\omega(0) = 1, \omega'(0) = -1, \omega^{(i)}(0) =0$ for $2\le i\le 4$.

We remark that the ideas from this and the previous section may be used to find differential equations for more general families of patterns, and also for permutations avoiding more than one pattern.
For more work in this direction, see~\cite{KS}.

\section{The pattern $1324$ and generalizations}
\label{sec:1324}

\subsection{The pattern $1324$}

This pattern has been considered in~\cite{Dot,LR}. In~\cite{Dot}, Dotsenko and Khoroshkin give a recurrence for its cluster numbers $\cl_{n,k}$. This recurrence, which involves the Catalan numbers, is essentially equivalent
to our derivation of equation~\eqref{eq:Clo1324} below. In~\cite{LR},
Liese and Remmel use a technique developed in~\cite{MR} to obtain an ordinary generating function that is equivalent to $\Clo_{1324}(-1,x)$.
Here we find the differential equation satisfied by the bivariate generating function $\omega_{1324}(u,z)$.

\begin{theorem}\label{thm:1324}
Let $\omega(z)=\omega_{1324}(u,z)$. Then $\omega$ is the solution of
$$
\begin{array}{ll}
& z\omega^{(5)}-((u-1)z-3)\omega^{(4)}-3(u-1)(2z+1)\omega^{(3)}+(u-1)((4u-5)z-6)\omega''+\\
&(u-1)(8(u-1)z-3)\omega'+4(u-1)^2z\omega=0,
\end{array}
$$
with $\omega(0) = 1, \omega'(0) = -1, \omega''(0) = \omega^{(3)}(0) = 0$.
\end{theorem}

\begin{proof}
In a cluster $(\pi;i_1,\dots,i_k)$ with respect to $1324$, we have $i_{j+1}-i_j\in\O_{1324}=\{2,3\}$ for all $j$.
Consider first $k$-clusters of length $n$ where $i_{j+1}-i_j=2$ for all $j$, i.e., $(i_1,\dots,i_k)=(1,3,5,\dots,2k-1)\in\I^\sigma_{n,k}$, where $n=2k+2$.
The poset in $\P^\sigma_{n,k}$ corresponding to this choice of indices is drawn in Figure~\ref{fig:poset1324} (when $k=1$, this poset is just a chain).
The number of linear extensions of this poset is the Catalan number $C_k=\frac{1}{k+1}\binom{2k}{k}$, since it equals the number of standard Young tableau of shape $2\times k$ (see for example~\cite{EC2}).

\begin{figure}[htb]
\centering
\psfrag{p1}{$\pi_1$}\psfrag{p2}{$\pi_2$}\psfrag{p3}{$\pi_3$}\psfrag{p4}{$\pi_4$}\psfrag{p5}{$\pi_5$}\psfrag{p6}{$\pi_6$}\psfrag{p7}{$\pi_7$}
\psfrag{pkk-2}{$\pi_{2k-2}$}\psfrag{pkk-1}{$\pi_{2k-1}$}\psfrag{pkk0}{$\pi_{2k}$}\psfrag{pkk1}{$\pi_{2k+1}$}\psfrag{pkk2}{$\pi_{2k+2}$}
\includegraphics[height=5cm]{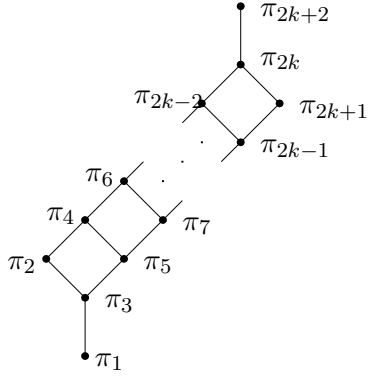}
\caption{\label{fig:poset1324} The order relationships in $k$-clusters with respect to $1324$ where neighboring occurrences overlap in two positions.}
\end{figure}

A cluster $(\pi;i_1,\dots,i_k)$ may contain two neighboring occurrences of $1324$ that overlap in only one entry, i.e., $i_{j+1}-i_j=3$ for some $j$. In this case, that entry $\pi_{i_{j+1}}$
is larger than all the entries of $\pi$ to its left and smaller than all the entries to its right.
In general, a poset in $\P^\sigma_{n,k}$ consists of a tower of pieces isomorphic to the poset in Figure~\ref{fig:poset1324},
where the top element of each piece is identified with the bottom element of the piece immediately above. A linear extension of the poset is uniquely determined by giving a linear extension for each one of the pieces,
since there are no incomparable elements in different pieces. It follows that the OGF for the cluster numbers is
\begin{align}\label{eq:Clo1324}\Clo_{1324}(t,x)=\frac{x}{1-\sum_{k\ge1} C_k t^kx^{2k+1}}-x=\frac{x}{1+x-x\,C(tx^2)}-x\\=\frac{x(1-2tx(1+x)+\sqrt{1-4tx^2})}{2(1-tx(1+x)^2)}-x,\nn\end{align}
where $C(x)=\sum_{k\ge0}C_kx^k=\frac{1-\sqrt{1-4x}}{2x}$.

To be able to apply Theorem~\ref{thm:GJ}, we need the exponential generating function for the cluster numbers. From equation~(\ref{eq:Clo1324}) and the fact that the generating function for the Catalan numbers satisfies $C(x)=1+xC(x)^2$,
we deduce that $\wh{A}(t,x)=1-x-\Clo_{1324}(t,x)$ satisfies the algebraic equation
$$(1-tx(1+x)^2)\wh{A}(t,x)^2+(2tx(1+x)+x-2)\wh{A}(t,x)-tx-x+1=0.$$
It follows that $\wh{A}(t,x)$ satisfies a linear differential equation, which we have found using the {\it Maple} package {\it gfun}:
\begin{multline*}
(4t^2x^5+8t^2x^4+(4t^2-t)x^3-6tx^2-tx+1)x\frac{\partial}{\partial x}\wh{A}(t,x)\\
+(4t^2x^5-4t^2x^3-2tx^3+6tx^2-1)\wh{A}(t,x)+
(4t^2+2t)x^3-7tx^2+1=0.\end{multline*}
Applying the operator $L$ to the above equation and using Lemma~\ref{lem:L}, we get
\begin{multline}\label{eq:IA}
(4t^2I^5+8t^2I^4+(4t^2-t)I^3-6tI^2-tI+1)z\frac{\partial}{\partial z}A(t,z)\\
+(4t^2I^5-4t^2I^3-2tI^3+6tI^2-1)A(t,z)+
(4t^2+2t)\frac{z^3}{6}-7t\frac{z^2}{2}+1=0.\end{multline}
Differentiating~(\ref{eq:IA}) four times with respect to $z$, we obtain a differential equation satisfied by
$A(z)=A(t,z)$, namely
$$zA^{(5)}-(tz-3)A^{(4)}-3t(2z+1)A^{(3)}+t((4t-1)z-6)A''+t(8tz-3)A'+4t^2zA=0,$$
with initial conditions
$A(0) = 1, A'(0) = -1, A''(0) = 0, A^{(3)}(0) = 0$.
Making the substitution $t=u-1$ and using Theorem~\ref{thm:GJ}, we obtain an equation for $\omega$.
\end{proof}

\subsection{The pattern $134\dots(s+1)2(s+2)(s+3)\dots m$}

The method that we used to find a differential equation satisfied by $\omega_{1324}(u,z)$ can be generalized to the pattern
$\sigma=134\dots(s+1)2(s+2)(s+3)\dots m$ for arbitrary $s\ge2$. Note that for $m\ge 2s$, $\sigma$ is a chain pattern and Corollary~\ref{cor:chain2} applies.
Thus, in this section we consider only the case $s+2\le m\le 2s$.

\begin{theorem}\label{thm:1324gen}
Fix $s+2\le m\le 2s$, and let $\sigma=134\dots(s+1)2(s+2)(s+3)\dots m$.
Then
\beq\label{eq:Clo1324gen}\Clo_\sigma(t,x)=\frac{x^{m-s}(B(tx^s)-1)}{1-(x+x^2+\dots+x^{m-s-1})(B(tx^s)-1)},\eeq
where $$B(x)=\sum_{k\ge0}\frac{1}{(s-1)k+1}\binom{sk}{k}x^k.$$
\end{theorem}

\begin{proof}
In any cluster $(\pi;i_1,\dots,i_k)$ with respect to $\sigma$,
we have $i_{j+1}-i_j\in\O_\sigma=\{s,s+1,\dots,m-1\}$ for all $j$. The overlaps that create incomparable elements in the poset are between neighboring occurrences of $\sigma$ that share $m-s$ entries.
Let us first consider $k$-clusters of length $n=(k-1)s+m$ where $i_{j+1}-i_j=s$ for all $j$, which we call {\em dense} clusters.
The poset in $\P^\sigma_{n,k}$ giving the order relationships in dense $k$-clusters is drawn in Figure~\ref{fig:poset1324gen}. The number of dense $k$-clusters is the number of linear extensions of this poset, which we call $Q_k$.
Note that this number does not change if we erase the bottom element and the top $m-s-1$ elements.
The linear extensions of the resulting poset $\widetilde{Q_k}$ are in bijection with lattice paths from $(0,0)$ to $(ks,0)$ with steps $u=(1,s-1)$ and $d=(1,-1)$ that do not go below the $x$-axis. We call these $s$-Dyck paths, and note that
$2$-Dyck paths are just standard Dyck paths.
Indeed, to construct the path corresponding to a given a linear extension, read the elements of $\widetilde{Q_k}$ in the increasing order given by the linear extension.
For each element that is circled in Figure~\ref{fig:poset1324gen}, draw an up-step $u$, and for each of the other elements, draw a down-step $d$. An example of this bijection is given in Figure~\ref{fig:posetbij1324}.

\begin{figure}[htb]
\centering
\psfrag{p1}{$\pi_1$}\psfrag{p20}{$\pi_2$}\psfrag{p30}{$\pi_3$}\psfrag{ps0}{$\pi_s$}\psfrag{ps1}{$\pi_{s+1}$}\psfrag{ps2}{$\pi_{s+2}$}
\psfrag{ps3}{$\pi_{s+3}$}\psfrag{p2s0}{$\pi_{2s}$}\psfrag{p2s1}{$\pi_{2s+1}$}\psfrag{p2s2}{$\pi_{2s+2}$}\psfrag{p2s3}{$\pi_{2s+3}$}
\psfrag{p3s1}{$\pi_{3s+1}$}
\psfrag{pk-2s2}{$\pi_{(k-2)s+2}$}\psfrag{pk-2s3}{$\pi_{(k-2)s+3}$}\psfrag{pk-1s0}{$\pi_{(k-1)s}$}\psfrag{pk-1s1}{$\pi_{(k-1)s+1}$}
\psfrag{pk-1s2}{$\pi_{(k-1)s+2}$}\psfrag{pk-1s3}{$\pi_{(k-1)s+3}$}\psfrag{pks0}{$\pi_{ks}$}\psfrag{pks1}{$\pi_{ks+1}$}
\psfrag{pks2}{$\pi_{ks+2}$}\psfrag{pks3}{$\pi_{ks+3}$}\psfrag{pk+1s}{$\pi_{(k-1)s+m}$}
\includegraphics[height=9cm]{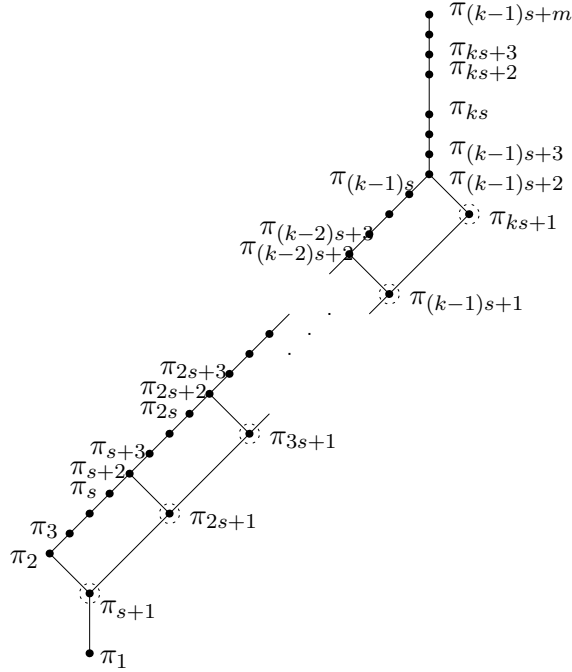}
\caption{\label{fig:poset1324gen} The poset $Q_k$ in $\P^{134\dots(s+1)2(s+2)(s+3)\dots m}$ corresponding to dense $k$-clusters. In this picture, $s=5$ and $m=10$.}
\end{figure}

\begin{figure}[htb]
\centering
\includegraphics[height=6cm]{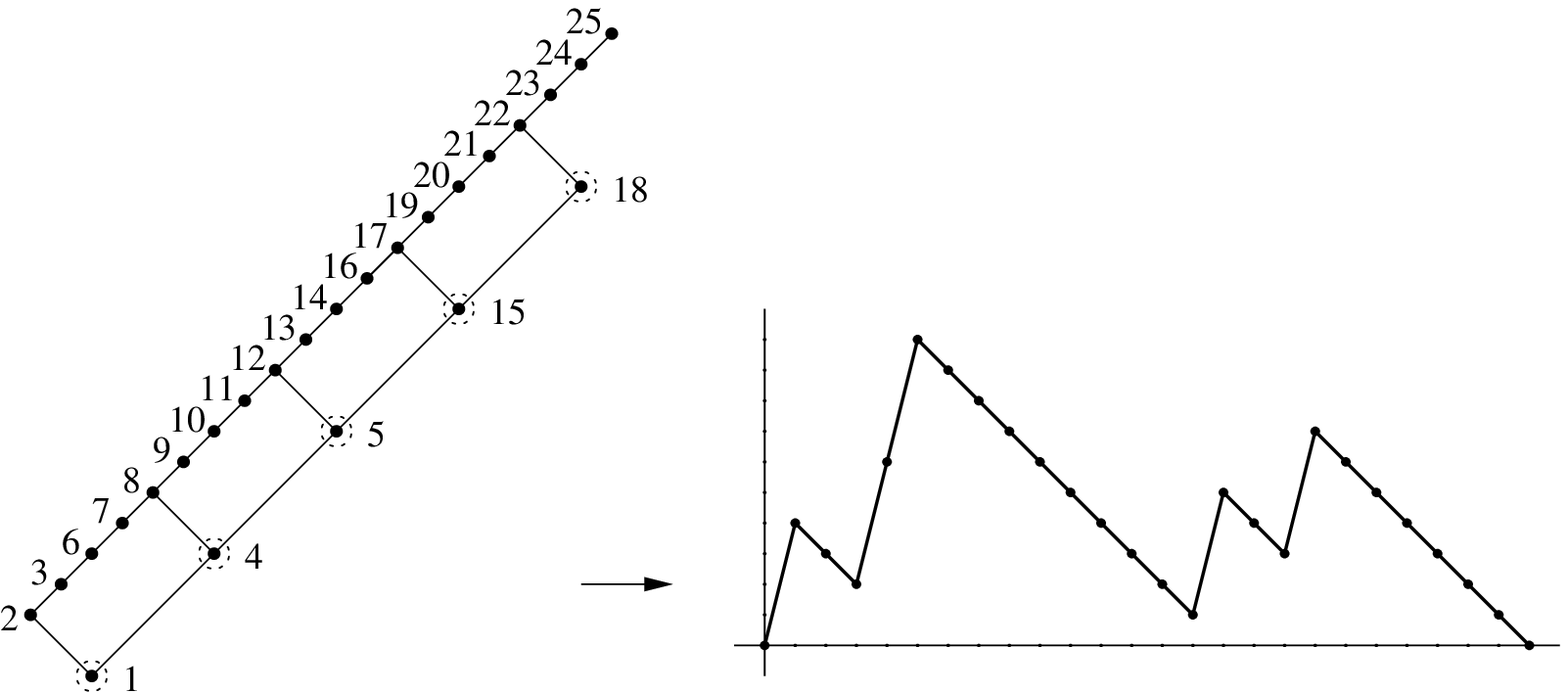}
\caption{\label{fig:posetbij1324} An example of the bijection between linear extensions of $\widetilde{Q_5}$ and lattice paths, for $s=5$ and $k=5$.}
\end{figure}

An $s$-Dyck path $L$ can be uniquely decomposed as $L=uL_1dL_2d\dots L_{s-1}dL_s$, where the $L_j$ are $s$-Dyck paths. It follows that
the ordinary generating function $B(x)$ for $s$-Dyck paths, where the exponent of $x$ is the number of up-steps, satisfies the equation
\beq\label{eq:Bs}B(x)=1+xB(x)^s.\eeq
It is a standard application of Lagrange inversion to deduce that the number of $s$-Dyck paths with $k$ up-steps is $$\frac{1}{(s-1)k+1}\binom{sk}{k}.$$
We conclude that the OGF for dense clusters with respect to $\sigma$, or equivalently, linear extensions of posets $Q_k$ where $k$ can vary,
is $x^{m-s}(B(tx^s)-1)$.

In general, since neighboring occurrences in a cluster can overlap in any number between $1$ and $m-s$ of positions,
every cluster consists of a sequence of dense clusters, where neighboring dense clusters overlap in any number between $1$ and $m-s-1$ of positions.
The corresponding poset consists of a tower of pieces isomorphic to $Q_k$ for some $k\ge1$, where the $j$ largest elements (for some $1\le j\le m-s-1$) of each piece are identified with the $j$ smallest elements of the piece immediately above.
Since there are no incomparable elements in different pieces, a linear extension of the resulting poset is determined by the linear extensions of the pieces. The OGF for the cluster numbers follows.
\end{proof}

From the expression in Theorem~\ref{thm:1324gen} we can proceed as we did for the pattern $1324$
to obtain the EGF for the cluster numbers: we use equations~(\ref{eq:Clo1324gen}) and~(\ref{eq:Bs}) to get an algebraic equation for $\wh{A}(t,x)=1-x-\Clo_\sigma(t,x)$,
which we first turn into a differential equation for $\wh{A}(t,x)$ and then, using Lemma~\ref{lem:L}, into one for $\omega_\sigma(u,z)=A(u-1,z)$.

\begin{example}
Let $\sigma=13425$, which is the case $s=3,m=5$. From Theorem~\ref{thm:1324gen}, we obtain using Maple that $\wh{A}(t,x)=1-x-\Clo_\sigma(t,x)$ satisfies a differential equation of the form
$$p_0(t,x)x^2\frac{\partial^2}{\partial x^2}\wh{A}(t,x)+p_1(t,x)x\frac{\partial}{\partial x}\wh{A}(t,x)+p_2(t,x)\wh{A}(t,x)+\sum_{i=1}^{8}c_i(t)x^i=0,$$
where the $p_i(t,x)$ and the $c_i(t)$ are polynomials, and the $p_i(t,x)$ have degree 11 in $x$.
In general, the order of this differential equation is at most $s-1$, by \cite[Theorem 6.4.6]{EC2}.
Applying the transformation $L$ and using Lemma~\ref{lem:L} we obtain the following equation for $A(t,z)$:
$$p_0(t,I)z^2\frac{\partial^2}{\partial z^2}A(t,z)+p_1(t,I)z\frac{\partial}{\partial z}A(t,z)+p_2(t,I)A(t,z)+\sum_{i=1}^{8}c_i(t)\frac{z^i}{i!}=0.$$
Differentiating 11 times with respect to $z$, we get a linear differential equation for $A(t,z)$, and hence also for $\omega_{\sigma}(u,z)$. It is a differential equation of order 13 with polynomial coefficients.

It would be interesting to determine the smallest order of a differential equation satisfied by $w_\sigma(u,z)$ for arbitrary values of $s$ and $m$.
\end{example}

\section{Other patterns of length $4$}\label{sec:four}

\subsection{The pattern $1423$}\label{sec:1423}

As in the case of the pattern $1324$, we have that $\O_{1423}=\{2,3\}$.
In this case, for each $k\ge1$ there is a unique $k$-cluster $(\pi;i_1,\dots,i_k)$ of length $n=2k+2$ where $i_{j+1}-i_j=2$ for all $j$, because the poset $P^{1423}_{n,i_1,\dots,i_k}$ is a chain  $\pi_1<\pi_3<\pi_5<\dots<\pi_{2k+1}<\pi_{2k+2}<\dots<\pi_4<\pi_2$.

A general $k$-cluster consists of blocks of marked occurrences that overlap in two positions as above,
where the last occurrence in each block overlaps the first occurrence of the next block in one position (so that $i_{j+1}-i_j=3$).
If the first such block has $k_1$ occurrences of $1423$, then,
in the corresponding poset, the block forms a chain with $2k_1+2$ elements.
The element just above the middle of the chain is $\pi_{2k_1+2}$, which is also the first entry of the second block, and thus the bottom element of another chain with $2k_2+2$ elements.
The poset of order relationships satisfied by the entries of the cluster is drawn in Figure~\ref{fig:poset1423}.

\begin{figure}[htb]
\centering
\psfrag{k1.}{$k_1$}\psfrag{k2.}{$k_2$}\psfrag{k3.}{$k_3$}\psfrag{k1+1}{$k_1+1$}\psfrag{k2+1}{$k_2+1$}\psfrag{k3+1}{$k_3+1$}
\includegraphics[height=7cm]{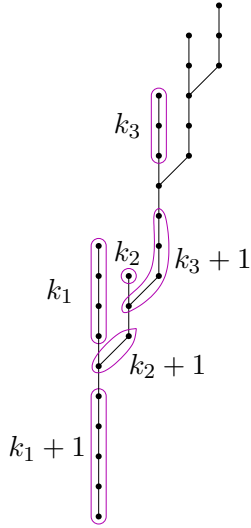}
\caption{\label{fig:poset1423} A generic poset in $\P^{1423}$.}
\end{figure}

It follows that the cluster numbers satisfy the recurrence
\beq\label{eq:rec1423}\cl_{n,k}=\sum_{i=2}^{n/2}\binom{n-i-1}{i-1}\cl_{n-2i+1,k-i+1}\eeq
with initial condition $\cl_{1,0}=1$ and $\cl_{i,j}=0$ for $i\le3$ in all other cases.
This recurrence was found by Dotsenko and Khoroshkin~\cite{Dot}.
Multiplying~\eqref{eq:rec1423} by $(-1)^k$ on both sides, summing over all $k$, and letting $s_n=\sum_k(-1)^k\cl_{n,k}$, we obtain the recurrence
\beq\label{eq:rec1423s}s_n=\sum_{i=2}^{n/2}(-1)^{i-1}\binom{n-i-1}{i-1}s_{n-2i+1},\eeq
with $s_1=1$ and $s_i=0$ for $i\le3$ in all other cases.

If we let $S(x)=1+\sum_{n\ge1}s_n x^n=1+x+\Clo_{1423}(-1,x)$, recurrence~\eqref{eq:rec1423s} is equivalent to the functional equation
\beq\label{eq:S}S(x)=1+\frac{x}{1+x}S\left(\frac{x}{1+x^2}\right).\eeq
Although it is straightforward to expand this equation to recover~\eqref{eq:rec1423s}, we had to use a generating tree with three labels for the set of clusters to find this equation.
Backed by numerical computations by Mireille Bousquet-M\'elou using the {\it Maple} package {\it gfun},
we conjecture that $S(x)$ is not D-finite, that is, it does not satisfy a linear differential equation with polynomial coefficients. This is equivalent to the following statement.

\begin{conjecture}\label{conj:nonDfinite}
The generating function $\omega_{1423}(0,z)$ is not D-finite.
\end{conjecture}

This conjecture is interesting for two reasons. On one hand, proving it would give the first known instance of a pattern $\sigma$ for which $\omega_{\sigma}(0,z)=1/P_\sigma(0,z)$ is not D-finite.
On the other hand, it suggests that a related conjecture for classical patterns (i.e., where occurrences are not restricted to consecutive positions) may be false.
Noonan and Zeilberger~\cite{NZ96}, extending a speculation of Gessel~\cite{Ges}, conjectured
that for any {\em classical} pattern, the generating function for the number of permutations avoiding it is D-finite. Note also that $P_\sigma(0,z)$ is already not D-finite for $\sigma=123$ (in the consecutive case),
since its denominator has a factor $\cos(3z/2+\pi/6)$ \cite{EliNoy}, hence $P_{123}(0,z)$ has infinitely many singularities.

\ms

Similar posets to the one in Figure~\ref{fig:poset1423} arise when considering the patterns $154263$ and $165243$. If $\sigma$ is either of these patterns, then
$\O_{\sigma}=\{3,5\}$, and for clusters where $i_{j+1}-i_j=3$ for all $j$,
$P^{\sigma}_{n,i_1,\dots,i_k}$ is a chain of length $n=3k+3$ where $\pi_{3k+3}$ has exactly $k+1$ elements strictly below it.
An argument analogous to the one for the pattern $1423$ implies that for every $n$ and $k$, each poset in $\P^{154263}_{n,k}$ is isomorphic to exactly one poset in $\P^{165243}_{n,k}$, and thus
$154263$ and $165243$ are strongly c-Wilf-equivalent. This proves another conjecture of Nakamura~\cite{Nak}.

\ms

We finish this section with a curious connection between occurrences of the pattern $1423$ and occurrences of a family of patterns.
The poset in Figure~\ref{fig:poset1423} suggests a close similarity between clusters for $1423$ and clusters for the infinite set of patterns $$\Sigma=\{12\dots a(a+2)(a+3)\dots(2a)(a+1):a\ge2\}=\{1243,123564,12346785,\dots\}.$$
Note that any two of these patterns can only overlap with each other in one position, and that in two adjacent occurrences, the `$a+1$' of the
left occurrence is the `$1$' of the right occurrence.

To be precise, suppose that a $k$-cluster with respect to $1423$ consists of blocks of $k_1, k_2, \dots, k_h$ occurrences
(with $k_j\ge1$ for all $j$),
where adjacent occurrences within a block overlap in two positions and the last occurrence of each block overlaps the first occurrence of the next block in one position. This poset, which is drawn in Figure~\ref{fig:poset1423}, is isomorphic to the poset of order relationships of an $h$-cluster with respect to $\Sigma$ where the $j$th pattern is
$12\dots k_j(k_j+2)(k_j+3)\dots(2k_j)(k_j+1)$. It follows that $\cl_{n,k}^{1423}=\cl_{n,h}^\Sigma$ when $k$ and $h$ are related by $2k+h+1=n$, and that $\cl_{n,h}^\Sigma=0$ if $n-h$ is even. Denoting $\cl_n^{1423}(t)=\sum_k \cl_{n,k}^{1423} t^k$ and $\cl_n^\Sigma(t)=\sum_h \cl_{n,h}^\Sigma t^h$, we have
$$\cl_n^\Sigma(t)=\sum_{\underset{n-h\ \mathrm{odd}}{h=1}}^{\frac{n-1}{3}}\cl_{n,h}^\Sigma t^h=\sum_{\underset{n-h\ \mathrm{ odd}}{h=1}}^{\frac{n-1}{3}}\cl_{n,(n-h-1)/2}^{1423}\,t^h=\sum_{k=\frac{n-1}{3}}^{\frac{n-2}{2}}\cl_{n,k}^{1423}t^{n-2k-1}=t^{n-1}\cl_n^{1423}(\frac{1}{t^2})$$
if $t\ne0$, so
$$\Cl_\Sigma(t,z)=\sum_n\cl_n^\Sigma(t)\frac{z^n}{n!}=\frac{1}{t}\sum_n\cl_n^{1423}(\frac{1}{t^2})\frac{(tz)^n}{n!}=\frac{1}{t}\Cl_{1423}(\frac{1}{t^2},tz).$$

In particular, the generating functions enumerating occurrences of these patterns are related by
$$P_\Sigma(u,z) =\frac{u-1}{u-2 +\dfrac{1}{P_{1423}\left(\frac{1}{(u-1)^2}+1,(u-1)z\right)}}.$$
It follows, for example, that $P_\Sigma(2,z)=P_{1423}(2,z)$, and that $\omega_\Sigma(0,z)+\omega_\Sigma(2,-z)=2$.

\subsection{The pattern $2143$}\label{subsec:2143}

Again we have that $\O_{2143}=\{2,3\}$, and so $k$-clusters with respect to $2143$ can be broken into blocks
in such a way that inside each block, adjacent occurrences overlap in two positions, and each block overlaps with the next in one position.
As in the case of the pattern $1423$, the poset corresponding to a block with $k_j$ occurrences is a chain of length $2k_j+2$. For example, if $k_1=i-1$, we get the chain $\pi_2<\pi_1<\pi_4<\pi_3<\dots<\pi_{2i}<\pi_{2i-1}$. The second largest element of each chain is the second smallest element of the next chain, corresponding to the position where a block overlaps with the next one (see Figure~\ref{fig:poset2143}).

\begin{figure}[htb]
\centering
\psfrag{C}{$C$}\psfrag{2k1+2}{$2i$}\psfrag{y1}{$\pi_{2i-1}$}\psfrag{y2}{$\pi_{2i}$}\psfrag{y3}{$\pi_{2i+1}$}\psfrag{P'}{$P'$}\psfrag{pi1}{$\pi_1$}\psfrag{pi2}{$\pi_2$}
\includegraphics[height=7cm]{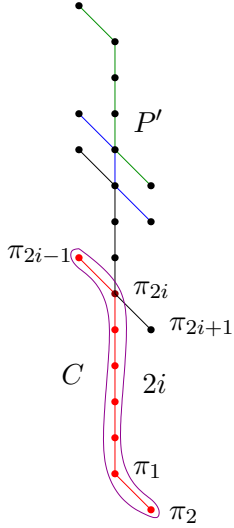}
\caption{\label{fig:poset2143} A generic poset in $\P^{2143}$.}
\end{figure}

The smallest element of each chain has the property that it does not cover any other elements in the poset. The elements other than the bottom one with this property
will be called {\em feet}. Compare this poset with the one in Figure~\ref{fig:poset1423}, which had no feet.
Let $\cl_{n,k,\ell}$ be the number of $k$-clusters of length $n$ where $\pi_1=\ell+2$. Such a cluster corresponds to a linear extension of a poset where exactly $\ell$ feet precede $\pi_1$.

For these refined cluster numbers, one obtains the recurrence
\beq\label{eq:rec2143}\cl_{n,k,\ell}=\delta_{n,k,\ell}+\sum_{i=2}^{n/2}\sum_{j=\ell-1}^{k-2}(n-2i-j)(\ell+1)\binom{2i+j-\ell-3}{2i-4}\cl_{n-2i+1,k-i+1,j},\eeq
where $\delta_{n,k,\ell}$ equals $1$ if $n=2k+2$ and $\ell=0$, and $0$ otherwise.
An equivalent recurrence was given by Dotsenko and Khoroshkin~\cite{Dot}.
To obtain~\eqref{eq:rec2143}, note that the poset $P$ in Figure~\ref{fig:poset2143} is either a chain with $n=2k+2$ elements (and thus with no feet), or
it consists of a chain $C$ with $2i$ elements attached to the poset $P'$ corresponding to the
$(k{-}i{+}1)$-cluster $\pi_{2i}\pi_{2i+1}\dots\pi_n$. Given a linear
extension of $P'$ having $j$ feet preceding $\pi_{2i}$, to produce a linear extension of
$P$ having $\ell$ feet preceding $\pi_{1}$ we have $n-2i-j$ choices for the value of $\pi_{2i-1}$,
$\ell+1$ choices for the value of $\pi_2$ (since it has to be one of the $\ell+1$ values lower than $\pi_1$), and
$\binom{2i+j-\ell-3}{2i-4}$ ways to interleave the $2i-4$ elements of $C$ and the $j-\ell+1$ elements of $P'$ whose value is strictly between $\pi_1$ and $\pi_{2i}$.

\section{Asymptotic and analytic results}\label{sec:asymptotic}

\subsection{Analytic properties of $\omega_\sigma$}\label{sec:analytic}

Recall that
$$
\omega_\sigma(u,z)=\frac{1}{P_{\sigma}(u,z)}
$$
is the inverse of the generating function counting occurrences of $\sigma$. In all the cases we have been able to solve, $\omega_\sigma(0,z)$ is an entire function in the complex plane, so that the counting generating function
$P_\sigma(0,z)$ is a meromorphic function and its dominant singularity is the smallest positive zero of $\omega_\sigma(0,z)$. We do not know if this is always the case (below we show an example where $\omega_\sigma(2,z)$ is \emph{not} entire), but we can prove it in several interesting cases. Since $\omega_\sigma(u,z)=1-z- R_\sigma(u-1,z)$,
this is naturally related to the growth of the cluster numbers $r_{n,k}$. Recall that $\P^\sigma_n$ is set of all posets associated to clusters of size $n$ with respect to $\sigma$.

\begin{theorem}\label{thm:entiregen}
Let $\sigma\in\S_m$. Suppose that there exists $\alpha>0$ so that, for all $n$, all posets in $\P^\sigma_n$ contain a chain of length at least $\alpha n$. Then, for every  fixed $u\in\mathbb{C}$, $\omega_\sigma(u,z)$ is an entire function of~$z$.
\end{theorem}

\begin{proof}
A poset in $\P^\sigma_n$ is determined by the positions $(i_1,\dots,i_k)$ of the occurrences of $\sigma$ in a $k$-cluster. Since $i_1 < \cdots < i_k$, the total number of such tuples (where $k$ is not fixed) is at most $2^n$. By hypothesis, each poset has a chain of length $\alpha n$, and  the number of linear extensions is at most the number of ways of deciding the relative order of the $n -\alpha n$ elements not in the chain. Each such element has at most $n$ possible placements
relative to the elements of the chain and the other elements that have already been placed. This gives an upper bound $n^{n-\alpha n}$ for the number of linear extensions of each particular poset. Summing over all clusters of size $n$ with respect to $\sigma$ we obtain
$$
    \sum_k r_{n,k} \le 2^n n^{n-\alpha n}.
$$
Using the bound $n! \ge n^n/e^n$, we get
$$
    \left(\frac{\sum_k r_{n,k}}{n!}\right)^{1/n} \le \frac{2e}{n^\alpha},
$$
which tends to $0$ as $n$ goes to infinity.

To  prove that $R_\sigma(t,z)= \sum r_{n,k}t^k\frac{z^n}{n!}$ is entire for any $t\in \mathbb{C}$, it is enough to show that
$$\lim_{n\rightarrow\infty} \left(\frac{\sum_k\cl_{n,k}t^k}{n!}\right)^{1/n}=0.$$
Using that
$$
    \left|\sum_k r_{n,k} t^k \right| \le \left\{
    \begin{array}{ll}
        |t|^n \sum_k r_{n,k} & \hbox{ if } |t|\ge1; \\
        \sum_k r_{n,k}  & \hbox{ if } |t|\le1, \\
            \end{array}
    \right.
$$
the result follows from the previous bound.
\end{proof}

The above theorem applies trivially to chain patterns, and more generally to patterns with $\sigma_1=1$, as well as to non-overlapping patterns.

\begin{corollary}\label{cor:chain}
If $\sigma\in\S_m$ is a chain pattern, then for any fixed $u$, $\omega_{\sigma}(u,z)$ is an entire function of $z$.
\end{corollary}

\begin{proof}
Every poset in $\P^\sigma_n$ is a chain of length $n$, so Theorem~\ref{thm:entiregen} applies with $\alpha=1$.
\end{proof}

\begin{corollary}\label{cor:sigma1}
If $\sigma\in\S_m$ satisfies $\sigma_1=1$, then for any fixed $u\in\mathbb{C}$, $\omega_\sigma(u,z)$ is an entire function of $z$.
\end{corollary}

\begin{proof}
Let  $(\pi;i_1,\dots,i_k)$ be a $k$-cluster with respect to $\sigma$.
Since $\pi_{i_j}$ is the smallest element in the occurrence of the $\sigma$ starting in position $i_j$, we have $$\pi_{i_1}<\pi_{i_2}<\dots<\pi_{i_k}.$$
Since $\sigma$ has length $m$, this gives a chain in the associated poset of length $k\ge (n-1)/(m-1)\ge n/m$, so Theorem~\ref{thm:entiregen} applies taking $\alpha=1/m$.
\end{proof}

\begin{corollary}\label{cor:entireNonoverlap}
For any non-overlapping pattern $\sigma\in\S_m$ and for every fixed $u\in\mathbb{C}$, $\omega_\sigma(u,z)$ is an entire function of $z$.
\end{corollary}

\begin{proof}
Let $a=\sigma_1$ and $b=\sigma_m$, and note that $a\neq b$. Without loss of generality we can assume that $a<b$, since $\sigma$ and its reversal have the same cluster numbers. We know from Section~\ref{sec:non-overlapping} that $k$-clusters with respect to $\sigma$ have length $n=k(m-1)+1$, and Figure~\ref{fig:poset1overlap}(ii) shows that every poset in $\P^\sigma_n$ contains a chain of length $m+(b-a)(k-1)$.
Since $b-a\le m-1$, we have that $$m+(b-a)(k-1)\ge 1+(b-a)k = 1+(b-a)\frac{n-1}{m-1} \ge \frac{b-a}{m-1}\,n.$$
Thus, Theorem \ref{thm:entiregen} applies with $\alpha =(b-a)/(m-1)$.
\end{proof}

We conclude this section by showing that $\omega_\sigma(u,z)$ may fail to be entire.
Out of the seven equivalence classes for consecutive patterns $\sigma$ of length four,
five of them have $\sigma_1=1$, namely $1234$, $1324$, $1423$, $1342$ and $1243$. By Corollary \ref{cor:sigma1}, $\omega_{\sigma}(u,z)$ is always an entire function of $z$ for these patterns. For the pattern $\sigma=2143$, the argument in Section~\ref{subsec:2143} shows that every poset in $\P^{2143}_n$ contains a chain of length at least $n/3$, hence $\omega_{2143}(u,z)$ is entire as well by Theorem \ref{thm:entiregen}.
Perhaps surprisingly, this is not true  for the remaining length four pattern $2413$.

\begin{proposition}
The function $\omega_{2413}(2,z)$ is not entire.
\end{proposition}

\begin{proof}
Note that $\O_{2413}=\{2,3\}$, and consider clusters $(\pi;i_1,\dots,i_{2\ell})$ with respect to $2413$ where
$i_{j+1}-i_j$ equals $2$ for odd values of $j$, and it equals $3$ for even values of $j$, that is, $(i_1,i_2,\dots)=(1,3,6,8,11,13,16,\dots,5\ell-4,5\ell-2)$.
Such a cluster has size $n=5\ell+1$, and the poset $P^{2413}_{n,i_1,\dots,i_{2\ell}}$ is drawn in Figure~\ref{fig:poset2413}. We claim that this poset has at least $\ell!^5$ linear extensions,
since the poset can be partitioned into 5 horizonal layers (4 of size $\ell$ and one of size $\ell+1$) so that the elements within each layer are incomparable. It follows that for $n=5\ell+1$,
$$\cl_{n,2\ell}\ge |\lext(P^{2413}_{n,i_1,\dots,i_{2\ell}})|\ge \ell!^5\approx\frac{n^n}{(5e)^n},$$
using Stirling's approximation.
Hence, disregarding polynomial terms in $n$,
$$
    \left(\frac{\sum_k r_{n,k}}{n!}\right)^{1/n} \ge \frac{1}{5},
$$
and the radius of convergence of $R_{2413}(1,z)$ is finite. Thus $\omega_{2413}(2,z) = 1-z-R_{2413}(1,z)$ is not an entire function.
\end{proof}

\begin{figure}[htb]
\centering
\psfrag{p1.}{$\pi_1$}\psfrag{p2}{$\pi_2$}\psfrag{p3}{$\pi_3$}\psfrag{p4}{$\pi_4$}\psfrag{p5.}{$\pi_5$}\psfrag{p6}{$\pi_6$}\psfrag{p7}{$\pi_7$}\psfrag{p8}{$\pi_8$}\psfrag{p9}{$\pi_9$}\psfrag{p10}{$\pi_{10}$}\psfrag{p11}{$\pi_{11}$}
\psfrag{p5l.}{$\pi_{5\ell}$}\psfrag{p5l-1}{$\pi_{5\ell-1}$}\psfrag{p5l-2}{$\pi_{5\ell-2}$}\psfrag{p5l-3}{$\pi_{5\ell-3}$}\psfrag{p5l-4}{$\pi_{5\ell-4}$}\psfrag{p5l+1}{$\pi_{5\ell+1}$}
\includegraphics[height=35mm]{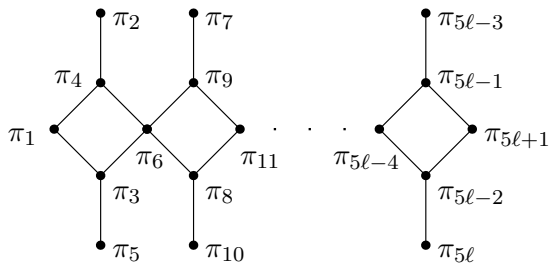}
\caption{\label{fig:poset2413} A poset in $\P^{2413}$.}
\end{figure}

Note that this argument does not disprove that $R_{2413}(-1,z)$ (equivalently, $\omega_{2413}(0,z)$) is entire, and in fact it is possible that,
with all the cancelations that take place in $R_{2413}(-1,z)$, its coefficients are much smaller.

\subsection{Asymptotic behavior of $\alpha_n(\sigma)$}

The first result in the literature regarding the asymptotic behavior of the number of permutations avoiding an arbitrary consecutive pattern is the following.

\begin{proposition}[\cite{Eliasym}]\label{prop:lim} For every $\sigma\in\S_m$ with $m\ge3$, there exist constants $0<c<d<1$ such that $c^n n!<\alpha_n(\sigma)<d^n n!$ for all $n$. Additionally, $\lim_{n\rightarrow\infty}\left({\alpha_n(\sigma)}/{n!}\right)^{1/n}$ exists, and it is strictly between $0.7839769$ and $1$.
\end{proposition}
\noindent
The value of this limit, which is called the growth constant of the pattern and is sometimes denoted~$\rho_\sigma$,
was determined in \cite{EliNoy} for  several patterns.
The following conjecture, which is equivalent to the fact that the monotone pattern has the largest growth constant among patterns of the same length, is still open.
\begin{conjecture}[\cite{EliNoy}]\label{conj:elinoy}
For every $\sigma\in\S_m$ there exists $n_0$ such that $\alpha_n(\sigma)\le\alpha_n(12\dots m)$ for all $n\ge n_0$.
\end{conjecture}
Our main result in this section is a proof of this conjecture
for non-overlapping patterns. We remark that B\'ona~\cite{Bon} has shown that the proportion of non-overlapping patterns is greater than $0.364$. Hence
we verify the conjecture for a positive fraction of all patterns of each length.

\begin{theorem}\label{thm:monvsnonover}
Let $m\ge3$ and let $\sigma\in\S_m$ be a non-overlapping pattern. Then there exists $n_0$ such that
$$\alpha_n(\sigma)<\alpha_n(12\dots m)$$ for all $n\ge n_0$.
\end{theorem}

\begin{proof}
We prove that the growth constants for these patterns satisfy $\rho_{\sigma}<\rho_{12\dots m}$. This is clearly equivalent to the statement of the theorem.

Let $\omega_{mon}(z)=\omega_{12\dots m}(0,z)$ and $\omega_{nol}(z)=\omega_{\sigma}(0,z)$. These are both entire functions by Corollaries~\ref{cor:chain} and~\ref{cor:entireNonoverlap}. It follows that the singularities of
the generating functions
\beq\label{eq:Pmonover} \frac{1}{\omega_{mon}(z)}=\sum_n\alpha_n(12\dots m)\frac{z^n}{n!} \quad \mathrm{and} \quad \frac{1}{\omega_{nol}(z)}=\sum_n\alpha_n(\sigma)\frac{z^n}{n!}\eeq
are poles at the zeroes of $\omega_{mon}(z)$ and $\omega_{nol}(z)$, and that their radii of convergence
are $1/\rho_{12\dots m}$ and $1/\rho_{\sigma}$, respectively.
Since the coefficients of these generating functions are non-negative, Pringsheim's Theorem~\cite[Theorem IV.6]{FS} implies  that they have real singularities at $z_{mon}=1/\rho_{12\dots m}$ and $z_{nol}=1/\rho_{\sigma}$. These
values are the smallest positive zeroes of $\omega_{mon}(z)$ and $\omega_{nol}(z)$, respectively. Our goal is to show that $z_{mon}<z_{nol}$. Since $\omega_{mon}(z)$ and $\omega_{nol}(z)$ are continuous functions on $\mathbb{R}$ with $\omega_{mon}(0)=\omega_{nol}(0)=1$, and we know by Proposition~\ref{prop:lim} that $1<z_{mon},z_{nol}<1.276$, it is enough to show that $\omega_{mon}(z)<\omega_{nol}(z)$ for $1<z<1.276$.

In the rest of the proof, we assume that $1<z<1.276$. For convenience, we  also assume that $m\ge4$, since the case $m=3$ is proved in~\cite{EliNoy}.
By equation~(\ref{eq:omega_monotone}),
$$\omega_{mon}(z)=\sum_{j\ge0}\frac{z^{jm}}{(jm)!}-\sum_{j\ge0}\frac{z^{jm+1}}{(jm+1)!}<1-z+\frac{z^m}{m!}-\frac{z^{m+1}}{(m+1)!}+\frac{z^{2m}}{(2m)!},$$
since each negative term of the alternating sum is larger in absolute value than the following positive term.
On the other hand, using that $k$-clusters with respect to $\sigma$ have length $n=k(m-1)+1$, and letting $d_k=r_{k(m-1)+1,k}$, we have
$$\omega_{nol}(z)=1-z-\sum_{k\ge1}(-1)^k d_k\frac{z^{k(m-1)+1}}{(k(m-1)+1)!}>1-z+\frac{z^m}{m!}-\sum_{\underset{k\ \mathrm{even}}{k\ge2}} d_k\frac{z^{k(m-1)+1}}{(k(m-1)+1)!},$$
since the $d_k$ are positive and $d_1=1$. Thus, it suffices to show that
\beq\label{eq:toshow} \sum_{\underset{k\ \mathrm{even}}{k\ge2}} d_k\frac{z^{k(m-1)+1}}{(k(m-1)+1)!}<\frac{z^{m+1}}{(m+1)!}-\frac{z^{2m}}{(2m)!}\eeq
for $1<z<1.276$ and $m\ge4$.

Let $a=\sigma_1$ and $b=\sigma_m$, and assume without loss of generality that $a<b$. When $k=2$, the poset in Figure~\ref{fig:poset1overlap}(i) has $d_2=\binom{a+b-2}{a-1}\binom{2m-a-b}{m-b}$ linear extensions. To find an upper bound on $d_2$, note that for any fixed value of $a+b$, the two binomial coefficients in this product are maximized when $b=a+1$. On the other hand, it is an exercise to
show that $\max_a \binom{2a-1}{a-1}\binom{2m-2a-1}{m-a-1}=\binom{2m-3}{m-2}$, and it is attained when $a=1$ or $a=m-1$. It follows that
$d_2\le \binom{2m-3}{m-2}$. A similar reasoning shows that for every $k$, $d_k$ is maximized when $a=1$ and $b=2$ (or, by symmetry, when $a=m-1$ and $b=m$). We obtain the bound
$$d_k\le \binom{k(m-1)-1}{m-2}\binom{(k-1)(m-1)-1}{m-2}\dots \binom{2(m-1)-1}{m-2}=\frac{(m-1)(k(m-1)-1)!}{(m-1)!^k (k-1)!},$$
the right hand side being the number of linear extensions of the poset in Figure~\ref{fig:poset1overlap}(i) when $a=1$ and $b=2$. This bound implies that
$$d_k\frac{z^{k(m-1)}}{(k(m-1))!}\le \frac{1}{k!}\left(\frac{z^{m-1}}{(m-1)!}\right)^k\le \frac{1}{2}\left(\frac{z^{m-1}}{(m-1)!}\right)^k,$$
for $k\ge2$, and so
\begin{multline*}\sum_{\underset{k\ \mathrm{even}}{k\ge2}} d_k\frac{z^{k(m-1)+1}}{(k(m-1)+1)!}\le \frac{z}{2m-1} \sum_{\underset{k\ \mathrm{even}}{k\ge2}} d_k\frac{z^{k(m-1)}}{(k(m-1))!}
< \frac{z}{2m-1}\sum_{\underset{k\ \mathrm{even}}{k\ge2}}\left(\frac{z^{m-1}}{(m-1)!}\right)^k\\
=\frac{z^{2m-1}}{(2m-1)(m-1)!^2\left(1-\frac{z^{2(m-1)}}{(m-1)!^2}\right)}.\end{multline*}
We have reduced the proof of inequality~\eqref{eq:toshow} to showing that
$$\frac{z^{m-2}}{(2m-1)(m-1)!^2\left(1-\frac{z^{2(m-1)}}{(m-1)!^2}\right)}<\frac{1}{(m+1)!}-\frac{z^{m-1}}{(2m)!}.$$
Since $z^{2(m-1)}<(m-1)!^2$, the above inequality is equivalent to
$$\frac{1}{(m+1)!}-\frac{z^{m-1}}{(2m)!}-\frac{z^{2(m-1)}}{(m-1)!^2(m+1)!}
+\frac{z^{3(m-1)}}{(m-1)!^2(2m)!}-\frac{z^{m-2}}{(2m-1)(m-1)!^2}>0$$
for $1<z<1.276$ and $m\ge4$. This is easy to verify, since the first term, which is positive,  dominates all the other terms.
\end{proof}

\section{Open problems}

In analogy with Conjecture~\ref{conj:elinoy}, Nakamura~\cite[Conjecture 2]{Nak} conjectures from numerical evidence that the pattern $123\dots(m-2)m(m-1)$ (which is non-overlapping) is
the hardest to avoid among all patterns of length $m$. A special case of this conjecture
is that this pattern is harder to avoid than any other non-overlapping pattern. It also appears to be the case that the pattern $134\dots m2$ is the easiest to avoid among non-overlapping patterns of the same length.
The last two conjectures can be combined as follows. We expect that the ideas in the proof of Theorem~\ref{thm:monvsnonover}, including a detailed analysis of the coefficients $d_k$, may be useful in proving this conjecture.

\begin{conjecture}
For every non-overlapping $\sigma\in\S_m$, there exists $n_0$ such that, for all $n\ge n_0$, $$\alpha_n(123\dots(m-2)m(m-1))\le\alpha_n(\sigma)\le\alpha_n(134\dots m2).$$
\end{conjecture}

\ms

We have seen in Section~\ref{sec:analytic} that $\omega_\sigma(u,z)$ is an entire function of $z$ for a large class of patterns. On the other hand, $\omega_{2413}(2,z)$ is not entire. However, it is still plausible
that $\omega_\sigma(0,z)$, which is the inverse of the generating function for permutations that avoid $\sigma$, is always an entire function.

\begin{question}\label{question:entire}
Is $\omega_\sigma(0,z)$ an entire function for every pattern $\sigma$?
\end{question}

\ms

To conclude, let us mention a recent result of Ehrenborg, Kitaev and Perry~\cite{EKP}, proved using methods from spectral theory, and previously conjectured by Warlimont~\cite{War}.

\begin{theorem}[\cite{EKP}]\label{thm:EKP}
For every pattern $\sigma$, $\alpha_n(\sigma)/n!=\gamma\rho^n+O(r^n)$, where $\gamma$, $\rho$ and $r$ are positive constants such that $\rho>r$.
\end{theorem}

It would be interesting to find a more combinatorial proof of this important result using singularity analysis of generating functions. If Question~\ref{question:entire} is answered positively, and one can show that the zero
of $\omega_\sigma(0,z)$ of smallest modulus is always unique and simple (we know that this is the case for several patterns), then the tools from \cite[Chapter IV]{FS} would give a proof of Theorem~\ref{thm:EKP} without using spectral theory.

\end{document}